\documentclass[a4paper,fleqn]{cas-sc}
\usepackage[numbers]{natbib}
\usepackage{lineno}
\usepackage{hyperref}
\usepackage{amsthm, amsmath, amssymb, amsfonts}
\newtheorem{theorem}{Theorem}[section]
\newdefinition{definition}{Definition}[section]                                                            
\def\tsc#1{\csdef{#1}{\textsc{\lowercase{#1}}\xspace}}
\tsc{WGM}
\tsc{QE}
\begin{document}
	\title [mode = title]{Impact of End Point Conditions on the Representation and Integration of Fractal Interpolation Functions and Well Definiteness of the Related Operator} 	
	\author[1]{Aparna M.P} 
	\fnmark[1]	
	\author[2]{P. Paramanathan}[orcid=0000-0003-0688-4858]
	\cormark[1]
	\fnmark[2]
	\address{Department of Mathematics, Amrita School of Engineering, Coimbatore,Amrita Vishwa Vidyapeetham,India.}
	\cortext[cor1]{Corresponding author}
	\fntext[fn1]{mp$\_$aparna@cb.students.amrita.edu (Aparna M.P); p$\_$paramanathan@cb.amrita.edu (P. Paramanathan)}
	\begin{abstract}
		Fractal interpolation technique is an alternative to the classical interpolation methods especially when a chaotic signal is involved. 
		The logic behind the formulation of an iterated function system for the construction of fractal interpolation functions is to divide the entire interpolating domain into subdomains and define functions on each subdomain piecewisely.  The objective of this paper is to explore the significance of the end point conditions on the graphical representation of the resultant functions and their numerical integration. The central problem in the formulation of the IFS, the continuity of the fractal interpolation functions, is addressed with an explanation on the techniques implemented to resolve the problem. Instead of an analytical expression, the fractal interpolation functions are always represented in terms of recursive relations. This paper further presents the derivation of these recursive relations and proposes a straightforward method to find the approximating function, involved in these relations. 
		\end{abstract}
	
	
	\maketitle
	\section{Introduction}
	\indent Fractal analysis deals with signals of chaotic behaviour. The chaotic nature of the signals can easily be identified by measuring the dimension of their graphs. The dimension must be non integer for such signals. Another indicator of the chaotic nature is that of self similarity. \\
	\indent In numerical approximation, when a signal is provided with only the sample values, there are numerous techniques available to find the interpolating function that passing through these values. But, when the signal is of self similar nature and its graph is of non integer dimension, different methods are to be applied. Thus, fractal interpolation functions (FIF) have been invented by Barnsley to approximate signals with such characteristics. The functions obtained via fractal interpolation are of non-differentiable nature. These characteristics enable them to represent most of the real world signals, whereas this representation is not possible with the usual interpolation functions (IF). The dimension of fractal interpolation function is a measure of the complexity of the signals, facilitating the comparison of recordings. Thus, fractal interpolation is considered as a generalisation of all the existing interpolation techniques.
	
	\indent The fractal interpolation functions are deeply connected with iterated function systems (IFS). Once a proper IFS has been defined, it is easier to derive the recursive formula for the fractal interpolation function. Fractal interpolation functions have been constructed for single variable functions, bivariable functions and multivariable functions. Although the fundamental aim is to create an IFS, there are some slight variations in the IFS according to the interpolating domain.
	\\ \indent The theory of fractal interpolation functions was put forward by Barnsley in \cite{ff}. 
	He introduced IFS and arrived at fractal interpolation functions as attractors of the IFS. Some of the properties of such interpolation functions have been explained in the paper and certain applications have been explored. The technique for the numerical integration using FIF was introduced by Navascues and Sebastian \cite{ni}. 
	The extension of the concept of fractal interpolation to two dimension was first carried out by Massopust \cite{fs}. 
	The construction was based on a triangular region with  interpolating points on the boundary are required to be coplanar. The coplanarity condition was later removed by Germino and Hardin in their work \cite{fi}. 
	They constructed fractal interpolation surfaces on triangles with arbitrary boundary data. The paper also considers the construction over polygons with arbitrary interpolation points. In \cite{bf}, L. Dalla considered rectangular interpolating domain where the interpolation points on the boundary were collinear. By using fold-out technique, Malysz introduced FIS over rectangles \cite{md}. The construction requires all the vertical scaling factors to be constant. Metzler and yun generalised this method by taking the vertical scaling factor as a function \cite{cf}. Hardin and Massopust constructed multivariable fractal interpolation functions \cite{fr}. A different construction with discontinuous FIS was considered in \cite{ifs}. In \cite{fsw}, Massopust constructed  bivariable fractal interpolation functions over rectangles by considering the tensor product of two univariable fractal interpolation functions. Xie and Sun considered another construction wherein the attractor was not the graph of a continuous function \cite{bfi}. The lack of continuity is later solved  by Vasileios Drakopoulos and Ploychronis Manousopoulos \cite{nt}. In \cite{nl}, Kobes and Penner considered non linear generalisation of fractal interpolation functions with an aim to apply in image processing fields.	Three dimensional IFS was considered in \cite{sf}. The construction of nonlinear fractal interpolation function is proposed in \cite{an}. In \cite{nn}, Ri and S.I introduced the construction of a bivariable FIF using Matkowski fixed point theorem and Rakotch contractions.\\
	\indent This paper is an attempt to investigate the impact of the end point conditions on the IFS for the construction and integration of fractal interpolation functions. Also, it explains the different techniques used to prove the well definiteness of the Read-Bajraktarevic operator whose fixed point is the required fractal interpolation function. This study begins with some basic definitions as a prerequisite to understand the rest of the sequel. A quick recap of the important steps in the construction is provided then. The recursive relation, that is to be satisfied by the fractal function has been proven logically and analytically. The next section presents the difference between an interpolation function and a fractal interpolation function for the same data set. The relation between the functions in the IFS and the approximating function involved in the recursive relation is established then. An easier method to find the approximating function is also provided. Following which, the significance of the end point conditions are explored and the consequences of violating them have been shown. Referring the construction procedure, it is observed that the fundamental problem in the construction of a bivariable fractal interpolation function is that of continuity. Since the Read Bejrakarevic operator is not well defined at the common boundaries of each subdomain, the resulting fixed point, i.e, the fractal interpolation function is not continuous. Finally, this paper concludes with an explanation on the different techniques implemented to solve the problem of well definiteness for a rectangular interpolating domain.
	\section{Prerequisites}
	The basic terminologies, definitions and results related to the topic are given below: 
	\begin{definition}
		Let $(X,d)$ be a metric space. Then, $(X,d)$ is a complete metric space if every cauchy sequences in $X$ converges in the metric space $(X,d).$ 
	\end{definition}
	\begin{definition}
		Let $(X,d)$ be a metric space. A function $f:X \rightarrow X$ is said to be a contraction map if, for any $x,y \in X,$ $$d(f(x),f(y)) \leq r. d(x,y)$$ for some $r$ such that $0 \leq r <1.$
	\end{definition}
	\begin{theorem}[The Contraction Mapping Theorem] \cite{fe}
		Let $f:X \rightarrow X$ be a contraction mapping on a complete metric space $(X,d)$. Then, $f$ possesses exactly one fixed point $x_{f} \in X$ and moreover for any point $x\in X$, the sequence $\{f^{on}(x): n=0,1,2,...\}$ converges to $x_{f}.$  That is,
		
		$$\lim\limits_{n \rightarrow \infty} f^{on}(x)=x_{f,}$$ for each $x \in X.$
	\end{theorem}

	\begin{definition}
		A hyperbolic iterated function system (IFS) consists of a complete metric space $(X,d)$ together with a finite set of contraction mappings $f_{n}:X \rightarrow X, n=1,2,...N$ for some $N \in \mathbb{N}.$ It is generally denoted as $\{X; f_{n}, n=1,2,...N\}.$ Its contractivity factor is $s=max\{s_{n}:n=1,2,..N\}$ where $s_{n}$ is the contractivity factor of $f_{n}.$
	\end{definition}
	\begin{definition}
		Consider a hyperbolic iterated function system $\{X; f_{n}, n=1,2,...N\}.$ An operator $F:H(X)\rightarrow H(X)$ such that $$F(B)=\cup_{n=1}^{N}f_{n}(B)$$ is known as Hutchinson operator.	
	\end{definition}
	\begin{theorem}
		Let $\{X; f_{n}, n=1,2,...,N\}$ be a hyperbolic IFS for the given set of data. Let $s=max\{s_{n}:n=1,2,..N\}$ be the contractivity factor of the hyperbolic IFS. The map $F:H(X) \rightarrow H(X)$ defined by $$F(B)=\cup_{n=1}^{N}f_{n}(B),$$  for all $B \in H(X)$ is a contraction mapping on the complete metric space $(H(X),h(d))$ with contractivity factor $s.$ $i.e,$ $$h(F(B),F(C)) \leq s h(B,C)$$ for all $B, C \in H(X).$ Its unique fixed point $A \in H(X)$ is such that $$A=F(A)=\cup_{n=1}^{N}f_{n}(A)$$ and is given by $$A=\lim\limits_{n \rightarrow \infty}F^{on}(B)$$ for any $B \in H(X).$ The unique, fixed point of the operator $F$ is known as the attractor of the IFS.	
	\end{theorem}
	
	\section{Quick Recap on the Steps Involved in the Construction of Single Variable Fractal Interpolation Functions}
	The construction of a fractal interpolation function involves mainly three steps. \\ 
	\noindent The first and foremost step is the construction of an iterated function system (IFS) for the data set.  \\
	\noindent The given data set is of the form $\{(t_{n},x_{n}):n=0,1,...,N\},$ for some $N \in Z, N \geq 2,$ where $t_{n}'s \in R$  are the input arguments and the output arguments $x_{n}'s$ are such that $x_{n}=f(t_{n})$ for every $n.$ Moreover, the input arguments are ordered by the usual ordering, $t_{0} <t_{1} < ...<t_{N}.$
	\\
	\noindent Given a data set, the following are the procedures to arrive at an IFS:
	\begin{enumerate}
		\item Specify the interpolation domain, $D,$ consisting of the data points. It is the set on which the fractal interpolation function is defined. For a single variable function, $D$ will be a closed interval. $D$ can be a triangle, rectangle, circle, polygon or any two dimensional shape for a two variable function. Likewise, $D$ will be an n-dimensional shape when an n-variable function is considered. The data set will be a subset of $D\times R.$ 
		\item Create contractions $L_{n}$ from $D$ into each of its subparts $D_{n},$ such that each $L_{n}$ is a homeomorphism and satisfy the end point conditions. The end point conditions depend upon the domain $D.$
		\item Choose a vertical scaling factor between -1 and 1.
		\item Define continuous functions $F_{n}:D \times R \rightarrow R$ such that it satisfies the end point conditions and contractive in the output argument. Here also, the end point conditions depend upon the domain $D.$ 
		Usually, $F_{n}$ is defined such that $F_{n}(t,x)=\alpha_{n}x+q_{n}(t), $ where $q_{n}$ is a continuous function on $D.$ If $g_{0}$ is an approximating function to the data set, there exists a connection between $g_{0}$ and $q_{n}.$ This relation is specifically explained later for a single variable function and an alternative method for finding the approximation function $g_{0}$ is provided.		
	\end{enumerate}
	\noindent Once an IFS has been defined, the next step is to make it hyperbolic. Since the IFS may not be hyperbolic in Euclidean metric, a new metric has been defined based on a real number $\theta$ on $I\times R$, equivalent to Euclidean metric. The IFS turns out to be hyperbolic once a proper $\theta$ has been chosen. The next and the final step into the fractal interpolation function is to establish the existence of a continuous function defined on interpolating domain such that the function interpolates the data and that the graph of the function is the attractor of the IFS. In order to prove the existence of such a function, the following points have to be verified. 
	\begin{itemize}
		\item Define a complete metric space that consists of the set of all continuous functions defined on the interpolating domain, satisfying some properties, depending on the domain. 
		\item Define an operator, called Read-Bajraktarevic operator, on the complete metric space and make it well defined. 
		\item Prove the contractivity of the operator.
		\item Verify the unique, fixed point of the operator interpolates the data.
		\item Verify graph of unique, fixed point of the operator is the attractor of the IFS.
	\end{itemize} 
	Then, the unique, fixed point of the operator is the required fractal interpolation function to the given data set.
	\begin{theorem}
	The fractal interpolation function satisfies the recursive relation $$f(t)=F_{n}({L_{n}^{-1}(t), f(L_{n}^{-1}(t))}), \,\,\,\, n=1,2,...,N, \,\, t \in D_{n}.$$
		\end{theorem}
\begin{proof}
	Let $\mathbb{F}$ be the set of all continuous functions $f:I \rightarrow R$ where $I=[t_{0},t_{N}]$ such that $f(t_{0})=x_{0}$ and $f(t_{N})=x_{N}.$ Trivially, $\mathbb{F}$ is a complete metric space with respect to the supremum metric. Now, define an operator $T:\mathbb{F} \rightarrow \mathbb{F}$ by $(Tf)(t)=F_{n}(L_{n}^{-1}(t), f(L_{n}^{-1}(t)) ).$ $T$ is continuous on each open interval $(t_{n-1}, t_{n}).$ To prove the continuity of $T$ at the end points of each subinterval, consider an arbitrary point $t_{n} \in I_{n}.$ Then, 
	\begin{align*}
		(Tf)(t_{n}) &=F_{n}(L_{n}^{-1}(t_{n}), f(L_{n}^{-1}(t_{n})) )\\
		&=F_{n}(t_{N}, f(t_{N}))\\
		&=F_{n}(t_{N},x_{N})\\
		&=x_{n}
	\end{align*}
	Since $t_{n}$ is also a point in $I_{n+1},$
	\begin{align*}
		(Tf)(t_{n}) &=F_{n+1}(L_{n+1}^{-1}(t_{n}), f(L_{n+1}^{-1}(t_{n})) )\\
		&=F_{n+1}(t_{0}, f(t_{0}))\\
		&=F_{n+1}(t_{0},x_{0})\\
		&=x_{n}
	\end{align*}
	Therefore, $T$ is continuous everywhere. To complete the proof of well definiteness, consider
	\begin{align*}
		(Tf)(t_{0}) &=F_{1}(L_{1}^{-1}(t_{0}), f(L_{1}^{-1}(t_{0})) )\\
		&=F_{1}(t_{0}, f(t_{0}))\\
		&=F_{1}(t_{0},x_{0})\\
		&=x_{0}
	\end{align*} and
	\begin{align*}
		(Tf)(t_{N}) &=F_{N}(L_{N}^{-1}(t_{N}), f(L_{N}^{-1}(t_{N})) )\\
		&=F_{N}(t_{N}, f(t_{N}))\\
		&=F_{N}(t_{N},x_{N})\\
		&=x_{N}.
	\end{align*} Now, it remains to show that $T$ is contractive. For that, let $f,g \in \mathbb{F}.$ Then, 
	
	\begin{align*}
		d(Tf(t), Tg(t)) &=|Tf(t)-Tg(t)|\\ &=|\alpha_{n}||f(L_{n}^{-1}(t))- g(L_{n}^{-1}(t))|\\ &\leq |\alpha_{n}| d(f,g).	
	\end{align*}
	Hence, $d(Tf,Tg)\leq \delta d(f,g),$ where $\delta=max_{n}\{|\alpha_{n}|\}<1.$
	Now, by contraction mapping theorem, $T$ has a unique, fixed point $f$ in $\mathbb{F}.$\\
	$i.e,$ $(Tf)(t)=f(t),$ which implies $f(t)=F_{n}({L_{n}^{-1}(t), f(L_{n}^{-1}(t))}), \,\,\,\, n=1,2,...,N, \,\, t \in D_{n}.$
	From the definition of $F_{n}$ used for a single variable function, it is evident that $f(t)=\alpha_{n}f(L_{n}^{-1}(t))+q_{n1}L_{n}^{-1}(t)+q_{n0}.$
	\vspace{0.2cm}\\
	\indent This relation can also be achieved logically by going through the IFS theory as follows: 
	\vspace{0.2cm}
	\\
	\indent The given initial data set was  $\{(t_{n},x_{n}):n=0,1,...,N\},$ such that $x_{n}=f(t_{n})$ for every $n.$ $i.e,$ $f$ is an interpolation function to this data set. Now, by applying the functions in the IFS, the data set is transformed into $\{ (L_{n}(t), F_{n}(t,x)): n=1,2,..,N\},$ for which also $f$ is an interpolation function. $i.e,$ $f(L_{n}(t))=F_{n}(t,x).$ Applying $L_{n}^{-1},$ this becomes $f(t)=F_{n}(L_{n}^{-1}(t), f(L_{n}^{-1}(t))).$ 
\end{proof}

\section{Comparison of IF and FIF for the same data set}
\begin{theorem}
Consider a set of data points $\{(t_{n}, x_{n}):n=0,1,..,N\}.$ Let $f$ be the classical interpolation function to this data set. Let the set be assigned with an IFS $\{(L_{n}(t),F_{n}(t,x)):n=1,2,...,N\}.$ The corresponding affine fractal interpolation function be $g.$ Then, the function $f$ is approximately equal to $g$   if there exists a scale vector $\alpha^{'}$ such that $$||f-g||_{\infty}\leq w_{f}(h)+\frac{2|\alpha^{'}|_{\infty}}{1-|\alpha^{'}|_{\infty}}||f||_{\infty}$$ where $h$ is the constant step length $h=t_{n}-t_{n-1}$ and $w_{f}(h)$ is the modulus of continuity of $f.$ On the converse, the function $g$ will always be an interpolation function to the data set.
\end{theorem}
\begin{proof}
	The first part of the theorem is proved in lemma 3.2 of \cite{ni}. For the converse part, consider an arbitrary point $t_{n}.$ Since the fractal interpolation function is the unique, fixed point of the operator $T$ defined as $$(Tf)(t) \\ =F_{n}(L_{n}^{-1}(t), f(L_{n}^{-1}(t))),$$
	\begin{align*}
		f(t_{n})&=(Tf)(t_{n}) \\ &=F_{n}(L_{n}^{-1}(t_{n}), f(L_{n}^{-1}(t_{n}))) \\ &= F_{n}(t_{N},x_{N}) \\ &=x_{n}
	\end{align*} Therefore, $f$ is an interpolation function to the data set. Hence the proof. 
\end{proof}

\section{Relation between $g_{0}$ and $q_{n}$}
\begin{theorem}
Let $\{(t_{n},x_{n}): n=0,1,...,N\}$ be the given data set with IFS $\{(L_{n}(t),F_{n}(t,x)):n=1,2,...,N\}.$ Let $f$ be the fractal interpolation function corresponding to this data set. Then, $f$ satisfies another recursive relation $$f(t)=g_{0}(t)+\alpha_{n}(f-r)oL_{n}^{-1}(t)$$ for $t\in I_{n}$ where $g_{0}$ is a piecewise linear function through the points $\{(t_{n},x_{n}): n=0,1,...,N\}$ and $r$ is the line passing through $(t_{0},x_{0}), (t_{N},x_{N}).$  Moreover, the functions $g_{0}$ and $q_{n}$ are related by the equation $q_{n}oL_{n}^{-1}(t)=g_{0}(t)-\alpha_{n}roL_{n}^{-1}(t)$ for $t\in I_{n}.$
\end{theorem}
\begin{proof}
By lemma 3.2 in \cite{ca},  $(Tf)(t)=g_{0}(t)+\alpha_{n}(f-r)oL_{n}^{-1}(t)$ for $t\in I_{n}$ where $g_{0}$ is a piecewise linear function through the points $\{(t_{n},x_{n}): n=0,1,...,N\}$ and $r$ is the line passing through $(t_{0},x_{0}), (t_{N},x_{N}).$ Since the fractal interpolation function is the unique, fixed point of $T,$ it is trivial that $$f(t)=g_{0}(t)+\alpha_{n}(f-r)oL_{n}^{-1}(t)$$ for $t\in I_{n}$ with $g_{0}$ and $r$ as specified above. The relation between $g_{0}$ and $q_{n}$ is also established in the lemma 3.2.
\end{proof}

\begin{theorem}
The function $g_{0}$ can also be calculated by solving the system of two equations given below.
\begin{align*}
a_{n}t_{n}+b_{n} &=x_{n} \\
a_{n-1}t_{n-1}+b_{n-1} &=x_{n-1} 
\end{align*}
\end{theorem}

\begin{proof}
Since $q_{n}$ is a linear function in $t,$ the functions $g_{0}, r$ will also be linear in $t.$ Then, writing $g_{0}(t)=a_{n}t+b_{n}, $ for some constants $a_{n}, b_{n},$
consider the given system of equations 	\begin{align}
	a_{n}t_{n}+b_{n} &=x_{n} \\
	a_{n-1}t_{n-1}+b_{n-1} &=x_{n-1}. 
\end{align} Subtracting (1) from (2) gives $a_{n}=\frac{y_{n}-y_{n-1}}{x_{n}-x_{n-1}}.$ \vspace{0.2cm}
\\ Substituting this in (1), gives $b_{n}=\frac{y_{n-1}x_{n}-y_{n}x_{n-1}}{x_{n}-x_{n-1}}.$ 
\vspace{0.2cm} \\ Then 
\begin{align*}
	g_{0}(t)  &=a_{n}t+b_{n} \\ &=\Big(\frac{y_{n}-y_{n-1}}{x_{n}-x_{n-1}}\Big)t+\Big(\frac{y_{n-1}x_{n}-y_{n}x_{n-1}}{x_{n}-x_{n-1}}\Big), 
\end{align*}
which is the piecewise linear function with vertices $(t_{n},x_{n}), $ for $n\in \{1,2,...,N\}.$
\end{proof}

\section{Significance of the end point conditions on $F_{n}$}
The iterted function systems for the single variable and bivariable fractal interpolation functions are given below:\\
\noindent Consider a single variable fractal interpolation function. The IFS for the function is $\{(L_{n}(t), F_{n}(t,x)): n=1,2,..,N\}.$ The role of $L_{n}$ is to contract the entire interpolating domain $D$. The interpolating domain is a closed interval when a single variable function is considered. It contracts the interval into its subintervals by assigning the end points of $I$ into the respective end points of $I_{n}.$ The second component function $F_{n}$ in the IFS, responsible for vertical scaling, normally consists of another function $q_{n}$ of the input arguments $q_{n}(t)=q_{n1}t+q_{n0}.$ The constants in $q_{n}$ are obtained by solving the end point conditions on $F_{n}, i.e,$ the conditions,  
$$	F_{n}(t_{0}, x_{0})=x_{n-1} \,\,\,\, F_{n}(t_{N}, x_{N})=x_{n} $$ for $n=1,2,...,N.$

When a two variable function with triangular region is considered, the IFS is $\{(L_{n}(x,y), F_{n}(x,y,z)): n=1,2,..,N\}$ where $L_{n}$ contracts the entire triangle into its subtriangles. The function $F_{n}$ is such that $F_{n}(x,y,z)=\alpha_{n}z+q_{n}(x,y)$ with $q_{n}(x,y)=a_{n}x+b_{n}y+c_{n}.$  The respective end point conditions are 
$$ 	F_{n}(x_{1}, y_{1}, z_{1})=z_{n1} \,\,\,\, F_{n}(x_{2}, y_{2},z_{2})=z_{n2}, \,\,\,\,F_{n}(x_{3}, y_{3},z_{3})=z_{n3} $$ for $n=1,2,...,N$ where $(x_{1}, y_{1}, z_{1}), \,\,\, (x_{2}, y_{2}, z_{2}), \,\,\, (x_{3}, y_{3}, z_{3}) $ are the vertices of the triangle $D$ with   $z_{n1}, \,\,\, z_{n2}, \,\,\, z_{n3}$ the value of the function at the vertices of the $n-th$ subtriangle. \\
\indent For rectangular region, the IFS is $\{(\phi_{n}(x), \psi_{m}(y), F_{n,m}(x,y)): n=1,2,..,N, m=1,2,...,M\}$ where $\phi_{n}$ contracts along the $X$ axis and $\psi_{m}$ contracts along the $Y$ axis. The function $F_{n,m}$ is such that $F_{n,m}(x,y,z)=\alpha_{n,m}z + q_{n,m}(x,y)$ where $ q_{n,m}(x,y)=e_{n,m}x+f_{n,m}y+g_{n,m}xy+k_{n,m}.$ Since it is a rectangle, four conditions are there for $F_{n,m}. i.e, $
$$ 	F_{n,m}(x_{0}, y_{0}, z_{0,0})=z_{n-1,m-1} \,\,\,\, F_{n,m}(x_{N}, y_{M},z_{N,M})=z_{n,m} $$ 
$$ 	F_{n,m}(x_{0}, y_{M} z_{0,M})=z_{n-1,m} \,\,\,\, F_{n,m}(x_{N}, y_{0},z_{N,0})=z_{n,m-1} $$ for $n=1,2,...,N \,\,\, m=1,2,...,M.$ 

In short, for all the shapes, the first component function in the IFS is responsible for the contraction of the respective interpolation domains. These functions carry out contractions by assigning the corner points of the domain $D$ into the respective corners of the subdomain $D_{n}.$ The number of end point conditions will depend on the number of corners of $D.$ The second component function involves vertical scaling. Here also, the end point conditions varies according to the interpolating domain $D.$  This component function consists of another function of the input arguments. The coefficients of the input arguments are obtained by solving the end point conditions.\\

\subsection{Consequences of removing any of the end point conditions on the second component function}
\begin{itemize}
\item If any one of the end point condition is removed, the expression for $q_{n}$ changes which affects the shape of the attractor.
\item In the formula for the numerical integration, numerator varies according to $q_{n},$ thereby making a difference in the integral value.   
\item With the change in $q_{n},$ the approximating function $g_{0}$ will be changed, affecting the error analysis techniques.
\end{itemize}

These points are specific to single variable functions, but also, relevant for two variable functions.

\section{Implementation of the Different Techniques used to Make $'T'$ Well Defined}
A set of data points is given with an IFS to the data and a freely chosen vertical scaling factor between -1 and 1. Before finding the fractal interpolation function to this data set, the existence of such a function has to be verified. The obtained function, defined from the interpolating domain to the set of real numbers, must be a continuous, interpolating function whose graph coincides with the attractor of the IFS. \\
\indent In order to establish the existence of such a function, a complete metric space consisting of the set of all continuous functions satisfying certain properties is created. The properties, used to verify the interpolating nature of the function, varies according to the interpolating domain.  Then, an operator T is defined on the space, which has to be shown contractive. Then, its unique fixed point is the required FIF. Usually, T is defined in terms of $F_{n,m}$, which is defined piecewisely on each subdomain. It is trivial that the boundaries of each subdomain is shared by the nearest subdomain. So, the operator T will be well defined only if it provides the same output along the boundaries. But, usually, T gives different outputs along the boundaries, which makes it a discontinuous map, which then implies that T is not contractive. But, there are various techniques available in order to transform T into a well defined map. Two such methods are discussed below. 
\subsection{Formulation of $G_{n,m}$ from the existing function $F_{n,m}$}
In \cite{nt}, authors have introduced a function $G_{n,m}$ and formulated a new IFS with $G_{n,m}.$ An explanation for the functional representation of $G_{n,m}$ is provided below:\\
Consider the data set $\{(x_{n},y_{m},z_{n,m}): n=0,1,...,N, m=0,1,..,M\}$ where $z_{n,m}=f(x_{n},y_{m}).$ Let the usual IFS be $w_{n,m}(x,y,z)=(\phi_{n}(x),\psi_{m}(y), F_{n,m}(x,y,z)), $ where the functions $\phi_{n}, \,\, \psi_{m}, \,\, F_{m,n}$ are defined as in \cite{bf}. Now, defining the operator $T$ to be $(Tf)(x,y)=F_{n,m}(\phi_{n}^{-1}(x),\psi_{m}^{-1}(y),f(\phi_{n}^{-1}(x),\psi_{m}^{-1}(y))),$ then consider the right vertical side of the subrectangle $I_{n} \times J_{m},$ ie, along the line $x=x_{n}, y_{m-1}\leq y \leq y_{m}.$ Since this side is shared by the subrectangle $I_{n+1}\times J_{m},$ the operator $T$ is defined in two ways:
Considering this line as a side of  $I_{n} \times J_{m},$
\begin{align*}
\nonumber 
(Tf)(x_{n},y) &=F_{n,m}(\phi_{n}^{-1}(x_{n}), \psi_{m}^{-1}(y), f(\phi_{n}^{-1}(x_{n}), \psi_{m}^{-1}(y))) \\
\nonumber 
&= F_{n,m}(x_{N}, \psi_{m}^{-1}(y), f(\phi_{n}^{-1}(x_{n}), \psi_{m}^{-1}(y)))
\end{align*}

If it is in $I_{n+1}\times J_{m},$ 
\begin{align*}
 (Tf)(x_{n},y) &=F_{n+1,m}(\phi_{n+1}^{-1}(x_{n}), \psi_{m}^{-1}(y), f(x_{N}, \psi_{m}^{-1}(y))) \\
&= F_{n+1,m}(x_{0}, \psi_{m}^{-1}(y), f(x_{0}, \psi_{m}^{-1}(y)))
\end{align*}
Hence, the operator $T$ have two different values along the line. Therefore, Define the operator $T$ such that  
\begin{align*}
\nonumber
(Tf)(x_{n},y)&=\frac{\Big[  F_{n,m}(x_{N}, \psi_{m}^{-1}(y), f(\phi_{n}^{-1}(x_{n}), \psi_{m}^{-1}(y))) + F_{n+1,m}(x_{0}, \psi_{m}^{-1}(y), f(x_{0}, \psi_{m}^{-1}(y)))	\Big]}{2}
\end{align*}

Since $T$ is piecewisely defined on each subrectangle $I_{m}\times J_{n},$ it must be defined as
$$(Tf)(x,y)=G_{n,m}(\phi_{n}^{-1}(x), \psi_{m}^{-1}(y), f(\phi_{n}^{-1}(x), \psi_{m}^{-1}(y)))$$   where 
\begin{align*}
\nonumber 
G_{n,m}(\phi_{n}^{-1}(x_{n}), \psi_{m}^{-1}(y), f(\phi_{n}^{-1}(x_{n}), \psi_{m}^{-1}(y)))
&=G_{n,m}(x_{N}, \psi_{m}^{-1}(y), f(x_{N}, \psi_{m}^{-1}(y))) 	
\end{align*}
Comparing the earlier expression, it is obtained that $$G_{n,m}(x_{N},y,z)=\frac{\Big[F_{n,m}(x_{N}, y, z) + F_{n+1,m}(x_{0}, y, z)\Big]}{2}.$$
Similarly, taking each side of the subrectangle, the expression for $G_{n,m}$ as in \cite{nt} can be reached.
The well definiteness of the operator $T$ is clearly established in \cite{nt} with this new IFS. 
\subsection{Making the data points on the boundary  to be coplanar}
Authors in \cite{bf} used the condition of collinearity of the data set to prove the well definitess of $T.$ This work is based on rectangular domain where the interpolation points on the boundary of each subrectangle is assumed to be collinear. Consider the collinear data sets $\{ (x_{0},y_{m}, z_{0,m}): m=0,1,..,M\},$ $\{ (x_{N},y_{m}, z_{N,m}): m=0,1,..,M\},$ $\{ (x_{n},y_{0}, z_{n,0}): n=0,1,..,N\},$ $\{ (x_{n},y_{M}, z_{n,M}): n=0,1,..,N\}.$ The collinearity implies that, for example, the set $\{ (x_{0},y_{m}, z_{0,m}): m=0,1,..,M\}$ is such that $x=x_{0}, \,\, y=(1-\lambda)y_{0}+\lambda y_{M}, z=\lambda z_{0,0} + (1-\lambda)z_{0,M},$ for all $\lambda \in [0,1].$ 
\\
Since the data is assumed to be collinear and the set $\mathcal{F}$ is taken to be the set of all continuous functions satisfying 
\begin{align*}
\nonumber
f(x_{0},(1-\lambda)y_{0}+\lambda y_{M})&=(1-\lambda)z_{0,0}+\lambda z_{0,M} \\
\nonumber 
f(x_{N},(1-\lambda)y_{0}+\lambda y_{M}) &=(1-\lambda)z_{N,0}+\lambda z_{N,M}
\\
\nonumber
f((1-\lambda)x_{0}+\lambda x_{N},y_{0}) &=(1-\lambda)z_{0,0}+\lambda z_{N,0} \\
\nonumber  
f((1-\lambda)x_{0}+\lambda x_{N},y_{M}) &=(1-\lambda)z_{0,M}+\lambda z_{N,M}	
\end{align*} 
\noindent Applying $T$ on the side of the subrectangle $I_{n}\times J_{m},$, $i.e,$ along $x=x_{n}, \,\, y_{m-1}\leq y \leq y_{m},$
\begin{align*}
\nonumber 
(Tf)(x_{n},y)&=(Tf)(x_{n}, (1-\lambda)y_{m-1}+\lambda y_{m}) \\
\nonumber
&=F_{n,m}(x_{N}, (1-\lambda)y_{0}+\lambda y_{M}, f(x_{N}, (1-\lambda)y_{0}+\lambda y_{M}))
\end{align*}
\noindent By expanding $F_{n,m}, $ the expression reaches at $$(Tf)(x_{n},y)= (1-\lambda) z_{n,m-1}+\lambda z_{n,m}$$
\\
Considering the line as a side of $I_{n+1}\times J_{m},$ the same expression can be achieved. Hence, the operator is well defined along the boundaries of each subrectangle.

\section{Conclusion}
This paper analyses the steps involved in the construction of a fractal interpolation function and provides a comparison on the difference between interpolation functions and fractal interpolation functions. From this study, it is observed that the continuity of the fractal interpolation function can be achieved whenever the Read-Bajraktarevic operator, used to formulate the fractal interpolation function is well defined. Two techniques to implement the well definiteness of the operator have been described in this work. The recursive relation, satisfied by the fractal interpolation function is established using Read-Bejraktaarevic operator. The relation is also proven logically. Analysing the end point conditions on the IFS, it is obtained that the removal any of these will affect the graphical as well as the integral results of fractal interpolation functions. It also affects the approximating function, involved in the recursive relation, since the function is depended on the IFS. While establishing this dependency, this paper also provides an alternative method to find the approximating function.

\end{document}